\documentclass[12pt]{amsart}

\usepackage{pacotes}
\usepackage[foot]{amsaddr}

\DeclareMathOperator{\Cpx}{Cpx}
\DeclareMathOperator{\Stc}{StCpx}
\DeclareMathOperator{\id}{id}

\title{Universal Families of Rayless Graphs}
\author[Aurichi]{Leandro Fiorini Aurichi}
\email{L. Aurichi: aurichi@icmc.usp.br}

\author[Pinto]{Guilherme Eduardo Pinto}
\email{G. E. Pinto: guipullus.gp@usp.br}
\address{ICMC-USP, São Carlos, SP, Brazil}
\date{}

\begin{document}

\begin{abstract}
We study the existence and cardinality of universal families for classes of rayless graphs.
 It is known, by a result of Diestel, Halin, and Vogler, that the class of countable rayless graphs does not admit a countable universal family, leaving open the precise complexity of this class.
 We prove that for every infinite cardinal $\kappa$, the class of rayless graphs of cardinality at most $\kappa$ admits a strongly universal family of size exactly $\kappa^+$, and that no smaller family can exist.
 This settles the problem for the countable case and extends uniformly to higher cardinalities.

We further investigate subclasses defined by forbidding subgraphs.
 When finitely many finite graphs are forbidden, the strong complexity remains $\kappa^+$, except in degenerate cases where it collapses to countable. 
In contrast, the class of countable rayless graphs when forbidding certain infinite graphs has a complexity that reaches its maximum possible value, the continuum.
 Finally, we establish that natural subclasses---including rayless trees, bipartite rayless graphs, graphs without even cycles, and graphs without infinite trails---retain the minimal strong complexity $\kappa^+$. 
These results provide a comprehensive characterization of universality in rayless graphs and highlight both its stability under restrictions and its sensitivity to specific obstructions.
\end{abstract}

\maketitle

\section{Introduction}

The study of universal structures has long been a fruitful theme in graph and model theory. 
Since Rado’s construction of a countable universal graph containing every countable graph as an induced subgraph in \cite{Rado}, universal objects and families have played a key role in understanding the complexity and structure of various graph classes.
 In many cases, however, the existence of a single universal graph is obstructed by structural limitations of the class under consideration.
A classical example is given by the rayless graphs, where it is well known that no countable universal graph exists. 
This raises a natural question: if no universal element exists, how large must a universal family be?

The notion of universality for families of graphs, together with the corresponding invariants of \emph{complexity} and \emph{strong complexity}, provides a systematic framework to address this question.
For a class of graphs $\mathcal C$, the complexity measures the minimal cardinality of a family within $\mathcal C$ such that every element of the class admits an immersion into some member of the family.
When asked for {\emph strong embeddings}, one obtains the notion of strong complexity.
Earlier work of Diestel, Halin, and Vogler \cite{Diest} established that the class of countable rayless graphs has no countable universal family, proving that its strong complexity is at least $\aleph_1$.
This leaves a wide range of possibilities between $\aleph_1$ and the continuum, the largest possible size of such a family.

The aim of this paper is to determine the exact strong complexity of rayless graphs across cardinalities.
We show that for every infinite cardinal $\kappa$, the class of rayless graphs of size up to $\kappa$ admits a strongly universal family of size exactly $\kappa^+$, and that no smaller family can exist.
This settles the problem for the countable case and generalizes it to arbitrary cardinalities.
Beyond the unrestricted class, we analyze rayless graphs with additional constraints, such as forbidden finite subgraphs and structural subclasses including trees, bipartite graphs, and graphs without even cycles or infinite trails.
We demonstrate that in all these cases, the strong complexity remains $\kappa^+$, except in certain degenerate cases where the complexity collapses to a smaller cardinal.
These cases encompass several cases of forbidden subgraphs studied in general classes, as those on \cite{Kom} and \cite{Che}

In particular, we identify a rank-1 countable rayless graph $K$ such that forbidding $K$ as a subgraph raises the complexity to at least the continuum, thereby showing that the spectrum of behaviors is genuinely broad.
On the other hand, for many natural classes of rayless graphs, the lower bound $\kappa^+$ is sharp.

This paper is organized as follows. 
In Section~\ref{sec:prel}, we introduce the definitions of universality, complexity, and the rayless hierarchy.
Section~\ref{sec:ray} establishes the exact value of the strong complexity for rayless graphs of arbitrary cardinality.
Section~\ref{sec:Fin} studies families obtained by forbidding finitely many finite subgraphs, while Section~\ref{sec:Infinite} considers infinite forbidden subgraphs, including an extremal example leading to maximal complexity.
Section~\ref{sec:Other} discusses other natural subclasses of rayless graphs, such as trees, bipartite graphs, and graphs without infinite trails, showing that the complexity remains minimal in these settings.

    
    


     
\section{Universal Families and Hierarchy}\label{sec:prel}
    Due to varying definitions, the first step in studying universality and complexity of families is to define what it means to say a family is universal in a class.
    Two types of immersions will be considered, from each of which distinct notions are derived.

    \begin{defn}
        An immersion of a graph $G$ into the graph $H$ is a graph morphism $\phi:G\to H$ that is injective on the vertices.

        An immersion $\phi:G\to H$ is said to be strong if it is an isomorphism between $G$ and the induced subgraph $H[\textup{im}\phi]$.
    \end{defn}

    Universality means that every element of the class can be immersed in it, even with the notion of immersion possibly varying.
    As remarked earlier, there is no universal element in the class of countable rayless graphs, in such cases we may study possible universal families, or subclasses.

    \begin{defn}
        Given a class of graphs $\mathcal C$, we say that a family $\mathcal F\subseteq\mathcal C$ is (strongly) universal if for every element of the class $G\in\mathcal C$ there is an element of the family $H\in\mathcal F$ to which there is an (strong) immersion $\phi:G\to H$.
    \end{defn}

    The question now asked is if there is such a family, how small can it be?
    This way we associated to each class that admits a universal family a cardinal, called its (strong) complexity.

    \begin{defn}
        Given a class of graphs $\mathcal C$ that admits an (strongly) universal family, its complexity is the following cardinal.
        \[\Cpx(\mathcal C)=\min\{|\mathcal F|:\mathcal F\subseteq\mathcal C\mbox{ is a universal family}\}.\]

        Its strong complexity is the following cardinal.
        \[\Stc(\mathcal C)=\min\{|\mathcal F|:\mathcal F\subseteq\mathcal C\mbox{ is a strongly universal family}\}.\]
    \end{defn}

    It is important to remark that, if both are well defined, the strong complexity of a class is at least the complexity of the class.
    Diestel, Halin and Vogler, in \cite{Diest}, have not only proven that the class of countable rayless graphs lacks a universal element, but that its complexity is at least $\aleph_1$.
    In fact, one can see a similar result for the class of connected rayless graphs of size up to $\kappa$, that will be called a $\kappa$-rayless graph, for any infinite cardinal $\kappa$, as a result of the rayless hierarchy introduced by Schmidt in \cite{Schm} in the following way.

    \begin{defn}
        A graph $G$ has rank $0$ if, and only if, it is finite.

        For $\alpha>0$, a graph has rank $\alpha$ if, and only if, it does not have a smaller rank and there is a finite subgraph $F\subseteq G$, called a kernel, such that every connected component of $G-F$ has rank smaller than $\alpha$.
    \end{defn}

    \begin{teo}[\cite{Schm}]A graph is rayless if, and only if, it has a rank.\end{teo}

        \begin{defn}
        For every cardinal $\kappa$ and ordinal number $\alpha$, the class of $\kappa$-rayless graphs is denoted by $\mathbf A_\kappa$ and the class of $\kappa$-rayless graphs of rank strictly less than $\alpha$ is denoted by $\mathbf A_{\kappa}^\alpha$.
    \end{defn}

    If $H\subseteq G$ are rayless, then $H$'s rank is at most $G$'s rank.

    \begin{rem}
        Every $\kappa$-rayless graph has rank strictly smaller than $\kappa^+$.
    \end{rem}



%
%
    

	In fact, one can see that, for each infinite cardinal number $\kappa$, if $\alpha<\kappa^+$ is a non zero ordinal number, then there is a rayless tree of size $\kappa$ and rank exactly $\alpha$.

    As a result, for any infinite cardinal $\kappa$, a universal family $\mathcal F$ of the class of $\kappa$-rayless graphs has an element of $\mathcal F$ with rank at least $\alpha$, for every $\alpha<\kappa^+$. 
    Since $\kappa^+$ is a regular cardinal, $\mathcal F$ cannot have less than $\kappa^+$ elements, concluding the following.

    \begin{prop}\label{prop:piso}
        Let $\kappa$ be any infinite cardinal. The complexities of the class of $\kappa$-rayless graphs lies in the following range:
        \[\kappa^+\leq\Cpx(\mathbf A_\kappa)\leq\Stc(\mathbf A_\kappa)\leq2^\kappa.\]
    \end{prop}

\section{The Complexity of Rayless Families}\label{sec:ray}
    Proposition \ref{prop:piso} establishes the range for the complexity. 
    In this section, we prove that the strong complexity is the minimal possible:
    \begin{teo}\label{thm:rayl}
        Let $\kappa$ be an infinite cardinal. Then $\Stc(\mathbf A_\kappa)=\kappa^+$.
    \end{teo}

    The proof is carried out by an inductive construction on the rank, showing that every rank has small complexity.
    We begin with an auxiliary result:

    \begin{prop}\label{prop:rayl}
        For every cardinal $\kappa$ and every non-zero ordinal number $\alpha<\kappa^+$,
        \[\Stc(\mathbf A_\kappa^\alpha)=\begin{cases}
            \aleph_0,\ \text{if }\alpha\mbox{ is a successor ordinal;}\\
            \mbox{cf}(\alpha),\ \text{if }\alpha\mbox{ is a limit ordinal.}
        \end{cases}.\]
    \end{prop}
    \begin{proof}
    	It is enough to see that there is a strongly universal family of the target cardinality, for each ordinal $\alpha$.
    	That is the case, since, for the limit ordinal case, the set of ranks of the elements of any universal family is cofinal in $\alpha$.
    	As for the successor ordinal case, the set of minimal separators of the elements of any universal family is infinite.
    	
        For $\alpha=1$, take the family of every connected finite graph $\mathcal F^{(1)}$ and take an enumeration $\mathcal{F}^{(1)}=\{F_n:n\in\mathbb N\}$.
        It is a countable strongly universal family of $\mathbf A_\kappa^{1}$.

        For $\alpha$ a limit ordinal, consider that for each $\beta<\alpha$ there is a strongly universal family $\mathcal F^{(\beta+1)}\subseteq\mathbf A_\kappa^{\beta+1}$ and take a cofinal sequence $\langle\beta_\gamma:\gamma<\mbox{cf}(\alpha)\rangle$ in $\alpha$.
        Define $\mathcal F^{\alpha}=\bigcup_{\gamma<\mbox{cf}(\alpha)}\mathcal F^{(\beta_\gamma+1)}$.

        For $\alpha=\beta+1$, consider a strongly universal family $\mathcal F^{\beta}$.
        For each $n\in\mathbb N$, we will define a $\kappa$-rayless graph of rank $\beta$ called $H_n^{\alpha}$, which will have minimal kernel set $F_n$.
        For each $n\in\mathbb N$ and each graph $H$, let $\mathcal I(F_n,H)$ be the set of all strong immersions $f:F_n\to H$, and consider $H_f$ the induced subgraph $H-\mbox{im}f$.

        The graph $H_n^{\alpha}$ will have $F_n$ as its minimal kernel set and the connected components are $\kappa$ many copies of $H_f$ for each $H\in\mathcal F^{\beta}$ and each $f\in\mathcal I(F_n,H)$.
        Also, the edges between $F_n$ and $H_f$ correspond to those between $\mbox{im}f$ and $H_f$ in $H$, as illustrated in Figure \ref{fig:const}

        \begin{figure}[ht]
        \centering
        \input{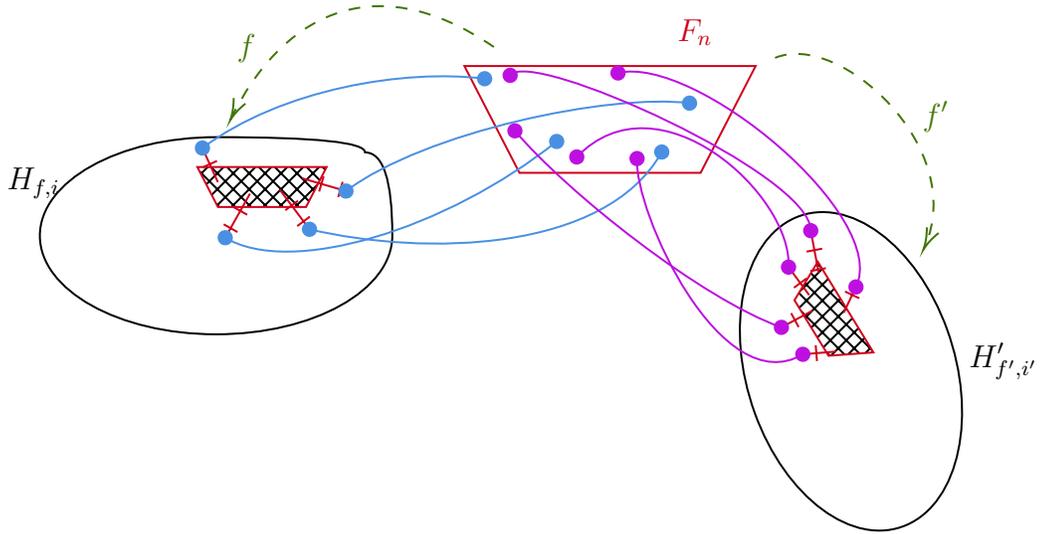}
        \caption{The construction}
        \label{fig:const}
        \end{figure}

        The formal construction is as follows.
        Let the vertex set be defined as
        \[V(H_n^{\alpha})=V(F_n)\cup\left(\bigcup\{V(H_f):H\in\mathcal F^{\beta},\ f\in\mathcal I(F_n,H)\}\times\kappa\right).\]

        There is an edge between $u,v\in V(H_n^{\alpha})$ if, and only if, one of the following holds:
        \begin{itemize}
            \item $u,v\in V(F_n)$ and $uv\in E(F_n)$;
            \item $u=(u^\prime,\lambda),v=(v^\prime,\lambda)\in V(H_f)\times\kappa$, for some $H\in\mathcal F^{\beta}$ and $f\in\mathcal I(F_n,H)$, and $u^\prime v^\prime\in E(H_f)$;
            \item $u\in V(F_n)$, $v=(v^\prime,\lambda)\in H_f\times\kappa$, for some $H\in\mathcal F^{\beta}$ and $f\in\mathcal I(F_n,H)$, and $f(u)v\in E(H)$.
        \end{itemize}

        It follows from the construction that $H_n^{\alpha}$ is a $\kappa$-rayless graph of rank $\beta$.
        Take $\mathcal F^{\alpha}=\{H_n^{\alpha}:n\in\mathbb N\}$, which is a strongly universal family.
        
        In order to verify this fact, let $G\in\mathbf A_\kappa^\alpha$.
        Consider a connected kernel $F\subseteq G$ and $n\in\mathbb N$ such that there is an isomorphism $\psi:F\to F_n$.
        Enumerate the family of all connected components of $G-F$ by $\{C_\lambda:\lambda<\mu\}$.
        For each $\lambda<\mu$, consider the induced subgraph $\overline C_\lambda=C_\lambda\cup F\subseteq G$.
        Consider also an element $H_\lambda\in\mathcal F^{\beta}$ with a strong immersion $\Psi_\lambda:\overline C_\lambda\to H_\lambda$.
        The restriction of $\Psi_\lambda$ to $F$ together with $\psi$ induces a strong immersion $f_\lambda\in\mathcal I(F_n,H_\lambda)$.

        Define the function $\Phi:V(G)\to V\left(H_n^{\alpha}\right)$ as
        \[\Phi(v)=\begin{cases}
            \psi(v)\in V(F_n);\text{ if }v\in V(F);\\
            (\Psi_\lambda(v),\lambda)\in V((H_\lambda)_{f_\lambda})\times \{\lambda\};\text{ if }v\in C_\lambda\text{, for }\lambda<\mu.
        \end{cases}.\]

        The function $\Phi$ defined is in fact a strong immersion from $G$ into $H_n^{\alpha}$.
        First, it is an injection by definition.
        Now, we will prove that, for each $u,v\in V(G)$, $uv\in E(G)$ if, and only if, $\Phi(u)\Phi(v)\in E(H_n^{\alpha})$.
        For that end, we will consider different cases for if $u$ and $v$ are in $F$ or $C_\alpha$.
        If both are at the same $\overline C_\lambda$, then the equivalency follows from the fact that $\Phi$ restricted to it is equivalent to $\Phi_\lambda$, which is a strong immersion.
        If $u$ and $v$ are not in the same $\overline C_\lambda$, then there is no edge between them.
        Moreover $\Phi(u)$ and $\Phi(v)$ are in different connected components of $H_n^{\alpha}$, thus, there is no edge between them.

        This concludes that the countable family constructed is a strongly universal family of $\mathbf A_\kappa^{\alpha}$.
    \end{proof}

    The proposition above gives the strong complexity of the classes of $\kappa$-rayless graphs with limited rank. It also allows us to construct a $\kappa$-rayless graph that contains every one of the graphs on the class as induced subgraphs, some times referred to as a not faithful universal graph. 

    \begin{cor}
        For every cardinal $\kappa$ and non-zero ordinal number $\alpha<\kappa$, there is a $\kappa$-rayless graph of rank $\alpha$ to which every $\kappa$-rayless graph of rank strictly less than $\alpha$ admits a strong immersion.
    \end{cor}

    Theorem \ref{thm:rayl} follows directly from the Proposition \ref{prop:rayl}, when considering for each $\alpha<\kappa^+$ a countable strongly universal $\mathcal F_{\alpha+1}$ of $\mathbf A_\kappa^{\alpha+1}$ and taking $\mathcal F=\bigcup_{\alpha<\kappa^+}\mathcal F_{\alpha+1}$, which is a strongly universal family of $\mathbf A_\kappa$.
    
\section{Forbidding Finite Subgraphs}\label{sec:Fin}
    The focus of most studies of the existence of universal graphs --- or the complexity of the class --- is on classes of graphs (up to a certain cardinality) with forbidden subgraphs, that is, the class of all graphs into which no graph of a chosen family can be immersed (see \cite{Diest}, \cite{KomPa},\cite{KomPaMe} and \cite{She}).
    The class of  rayless graphs is such an example.
    
    After having seen the complexity of the class of rayless graphs, a natural step is to investigate further restrictions.
    The first step into this investigation is studying finite restrictions, in both senses of the word, that is, forbidding finitely many finite subgraphs.

    \begin{defn}
        Let $G$ and $K$ be graphs.
        We say that $G$ is $K$-free if there is no immersion from $K$ to $G$.
        For a family of graphs $\mathcal K$, $G$ is $\mathcal K$-free if it is $K$-free for each $K\in\mathcal K$.

        Given an infinite cardinal number $\kappa$, the class of $\kappa$-rayless graphs that are $\mathcal K$-free is called the class of $\kappa$-rayless graphs with forbidden $\mathcal K$ and is denoted by $\mathbf A_\kappa(\mathcal K)$.
        For any $\alpha<\kappa^+$, denote $\mathbf A_\kappa^\alpha(\mathcal K)=\mathbf A_\kappa(\mathcal K)\cap\mathbf A_\kappa^{\alpha}$.
    \end{defn}

    In fact, what will be verified is that, for most cases, the complexity remains unchanged, and on the cases it is different, it is because the complexity has decreased.
    We begin by noting that, unless specific subgraphs are forbidden, the complexity remains $\kappa^+$.

    \begin{exemp}\label{exp:Kfree}
        If $\mathcal K$ is a finite family of finite graphs, none of which is a subdivision of the star (path included), then for every $\alpha<\kappa^+$ there is a graph $G\in\mathbf A_\kappa(\mathcal K)$ of rank $\alpha$.
    \end{exemp}
    \begin{proof}
        For each tree $K\in\mathcal K$, there is more than one vertex of degree greater than 2, since it is not a subdivision of a star.
        Define $l_K$ the smallest distance between two vertices of degree 3 or more of $K$, and take $l=\max\{l_K:K\in\mathcal K\mbox{ is a tree}\}$.

        For each $\alpha<\kappa^+$, consider $T_\alpha$ a $\kappa$-rayless tree of rank $\alpha$.
        Let $T_\alpha^\prime$ be the resulting tree of dividing each edge of $T_\alpha$ into $(l+1)$-paths.


    That $T_\alpha^\prime$ has size $\kappa$ and rank $\alpha$ is directly from the construction.
    Since it is a tree, if $K\in\mathcal K$ is not a tree, then it cannot be immersed in $T_\alpha$.
    Consider $K\in\mathcal K$ is a tree.
    Since every two distinct vertices of degree greater than 2 have distance at least $l+1>l_K$, then $K$ cannot be immersed in $T_\alpha^\prime$.
    Thus concluding that $T_\alpha^\prime$ is $\mathcal K$-free.
    \end{proof}

    \begin{rem}\label{rem:Kfree}
        Note that if there is a subdivision of a star on $\mathcal K$, then there is a maximal finite rank for the class of $\kappa$-rayless graphs with forbidden $\mathcal K$.
    \end{rem}

	Therefore, except for when $\mathcal K$ has a subdivision of a star, the complexity of the class $\mathbf A_\kappa(\mathcal K)$ is at least $\kappa^+$.
	In fact, the same argument is true when considering finitely many countable graphs, with subdivisions of countably infinite stars, except that the maximum rank is possibly $\omega$ instead of necessarily finite, depending whether there is a limit to the size of the subdivisions.

	For the finite case, we can show that this lower bound is sharp.

    \begin{teo}\label{thm:fin}
        Let $\kappa$ be an infinite cardinal and $\mathcal K$ a finite family of finite graphs.
        Then $\Stc(\mathbf A_\kappa(\mathcal K))=\kappa^+$, if no element of $\mathcal K$ is a subdivision of the star; and $\Stc(\mathbf A_\kappa(\mathcal K))\leq\aleph_0$ otherwise.
    \end{teo}

    In order to prove Theorem \ref{thm:fin}, we first lay some ground work, through some definitions and a lemma, the rest of the proof will follow the footsteps of Theorem \ref{thm:rayl}, together with its own version of Proposition \ref{prop:rayl}.

    \begin{defn}
        Given two graphs $K$ and $G$, a subgraph $F\subseteq G$ captures $K$ if for every immersion $\phi:K\to G$ there is an immersion $\phi^\prime:K\to F$ such that
        \begin{equation*}
            \mbox{If }v\in V(K)\mbox{ is such that }\phi(v)\in V(F)\mbox{, then }\phi^\prime(v)=\phi(v).
        \end{equation*}

        In this case, $\phi^\prime$ is the representative of $\phi$ in $F$.

        Given a family of graphs $\mathcal K$, we say that $F\subseteq G$ captures the parts of $\mathcal K$ if for each $K\in\mathcal K$ it captures every subgraph $\widetilde K$ of $K$. 
    \end{defn}

    It is important to note that capturing a graph $K$ is a transitive property, that is, if there is a chain of subgraphs $F\subseteq H\subseteq G$ such that $H\subseteq G$ captures a graph $K$ and $F\subseteq H$ also captures $K$, then $F\subseteq G$ captures $K$.
    This definition allows us to consider that, in a rayless graph, for any finite family of finite graphs $\mathcal K$, there is a kernel set such that the graph is $\mathcal K$-free if, and only if, the kernel is $\mathcal K$-free.

    \begin{lem}\label{lem:capture}
        Let $G$ be a rayless graph and $\mathcal K$ a finite family of finite graphs.
        For any finite subgraph $F\subseteq G$, there is a finite subgraph $\widetilde F\subseteq G$ such that $F\subseteq\widetilde F$ and it captures the parts of $\mathcal K$ on $G$.
        In particular $\widetilde F$ can be a kernel.
    \end{lem}
    \begin{proof}
        The proof of this result is  by induction on the rank of $G$.
        The case where $G$ is finite is trivial, due to the possibility of taking $\widetilde F=G$.
        Now, supposed that the statement is true for every graph of rank lower than the rank of $G$.
        Without loss of generality, suppose that $F\subseteq G$ is already a kernel set.
        Take $k$ the largest size of an element of $\mathcal K$.

        Consider $\mathcal I_0$ the family of restrictions to $F$ of immersions of $\phi:\widetilde K\to G$, for all elements $K\in\mathcal K$ and subgraphs $\widetilde K\subseteq K$. 
        For each $f\in\mathcal I_0$, take $\phi_f:\widetilde K\to G$ whose restriction is $f$.
        Define $H_0$ as the union of all (finitely many) connected components of $G-F$ that meet the image of any $\phi_f$.

        For each $i\leq k$, we will repeat a similar construction to achieve a pairwise disjoint family $H_0,..., H_k$ of unions of finitely many connected components of $G-F$.
        Given $i\leq k$, suppose that $H_0,...,H_{i-1}$ are already defined.
        Take $\overline H_i=H_0\cup...\cup H_{i-1}$.
        Define $\mathcal I_i$ similarly to $\mathcal I_0$, except that the image of $\phi$ is disjoint from $\overline H_i$ and define $H_i$ similarly to $H_0$, except that $\phi_f$ has its image disjoint from $\overline{H_i}$.

        The subgraph $H=F\cup H_0\cup...\cup H_k\subseteq G$ has strictly smaller rank, due to being the union of finitely many subgraphs of smaller rank.

        \begin{aff}
            $H\subseteq G$ captures the parts of $\mathcal K$.
        \end{aff}

        In order to prove this claim, take $\widetilde K\subseteq K\in\mathcal K$ and an immersion $\phi:\widetilde K\to G$.
        Since the size of $\widetilde K$ is at most $k$, there is an $i\leq k$ such that $H_i$ is disjoint from the image of $\phi$.
        Set $\overline \phi$ the restriction of $\phi$ to $\overline K=\widetilde K-\phi^{-1}[H-F]$.
        Thus, the restriction $f$ of $\overline\phi$ to $F$ was considered and there is $\phi_f:\overline K\to F\cup H_i$ that is representative of $\overline\phi$ in $F\cup H_i$.

        Take $\Phi:\widetilde K\to H$ as, for each $v\in V(\widetilde K)$, $\Phi(v)=\phi(v)$ if $\phi(v)\in H$, and $\Phi(v)=\phi_f(v)$ otherwise.
        It is a representative of $\phi$ on $H$, since by definition it preserves the points already in $H$ and it is an immersion, since $H_i$ is disjoint from the image of $\phi$, the restriction of $\phi$ to $\phi^{-1}[F]$ is $f$, the same as $\phi_f$ and both are immersions

        By induction hypothesis, there is $\widetilde F\subseteq H$ that contains $F$ and captures the parts of $\mathcal K$, therefore $\widetilde F\subseteq G$ captures the parts of $\mathcal K$.
    \end{proof}

    Now, in order to see Theorem \ref{thm:fin}, the universal families will be constructed as in Proposition \ref{prop:rayl}, with a modification in their formulation, from which will follow a property stronger than being universal.

    \begin{prop}\label{prop:Kfree}
        Let $\mathcal K$ be a finite family of finite graphs.
        For every cardinal $\kappa$ and non-zero ordinal number $\alpha<\kappa^+$ there is a family $\mathcal F_\kappa^{\alpha}(\mathcal K)\subseteq\mathbf A_\kappa^\alpha(\mathcal K)$ such that 
        \begin{enumerate}
        \item For every $G\in\mathbf A_\kappa^\alpha$ there is $H\in\mathcal F_\kappa^\alpha(\mathcal K)$ and a strong immersion $\Phi:G\to H$ whose image $\mbox{im}\Phi\subseteq H$ captures the parts of $\mathcal K$;\\
        \item The cardinality  is $|\mathcal F_\kappa^\alpha(\mathcal K)|=\begin{cases}
            \aleph_0,\ \mbox{if }\alpha\mbox{ is a successor ordinal;}\ \\
            \mbox{cf}(\alpha)\text{, if }\alpha\mbox{ is a limit ordinal.} 
        \end{cases}$.
        \end{enumerate}
    \end{prop}
    \begin{proof}
        For $\alpha=1$, take the family of every connected finite $\mathcal K$-free graph $\mathcal F_\kappa^1(\mathcal K)$ and an enumeration $\{F_n:n\in\mathbb N\}$.
        This family satisfies the desired properties.
        For a limit ordinal $\alpha$, it is enough to consider $\mathcal F_\kappa^\alpha(\mathcal K)=\bigcup_{\gamma<\mbox{cf}(\alpha)}\mathcal F_\kappa^{\beta_\gamma+1}(\mathcal K)$, for $\langle\beta_\gamma:\gamma<\mbox{cf}(\alpha)\rangle$ a cofinal sequence on $\alpha$.

        For $\alpha=\beta+1$. For each $n\in\mathbb N$, we will define a $\kappa$-rayless graph $\mathcal K$-free of rank $\beta$ called $H_n^\alpha$ with minimal kernel $F_n\subseteq H_n^\alpha$ that captures the parts of $\mathcal K$.

        Consider, for each $H\in\mathcal F_\kappa^\beta(\mathcal K)$, $\mathcal{I}(F_n, H;\mathcal K)$ the family of immersion $f:F_n\to H$ such that $\mbox{im}f\subseteq H$ captures the parts of $\mathcal K$.
        Define now the graph $H_n^\alpha$ analogously to how it is defined in Proposition \ref{prop:rayl}, except that $H_f$ is only considered if $f\in\mathcal I(F_n,H;\mathcal K)$.
        Let us remark that $H_n^\alpha\in\mathbf A_\kappa^\alpha(\mathcal K)$, since by construction it is a $\kappa$-rayless graph of rank up to $\beta$ and $F_n$, which is $\mathcal K$-free, captures the parts of $\mathcal K$ on $H_n^\alpha$. 
        Take $\mathcal F_\kappa^\alpha=\{H_n\alpha:n\in\mathbb N\}$ the family wanted.

        In order to verify the statement (1) of the Proposition, consider $G\in\mathbf A_\kappa^\alpha(\mathcal K)$.
        Consider a connected kernel $F\subseteq G$ that captures the parts of $\mathcal K$, and $n\in\mathbb N$ such that there is an isomorphism $\psi:F\to F_n$.
        Enumerate the family of all connected components of $G-F$ by $\{C_\lambda:\lambda<\mu\}$.
        For each $\lambda<\mu$, consider the induced subgraph $\overline C_\lambda=C_\lambda\cup F\subseteq G$.
        Consider also an element $H_\lambda\in\mathcal F_\kappa^\alpha(\mathcal K)$ with a strong immersion $\Psi_\lambda:\overline C_\lambda\to H_\lambda$ such that $\mbox{im}\Psi_\lambda\subseteq H_\lambda$ captures the parts of $\mathcal K$.

        Note that the immersion $f_\lambda=\Psi_\lambda\circ\psi^{-1}:F_n\to H_\lambda$ is in $\mathcal I(F_n,H_\lambda;\mathcal K)$, since $\mbox{im}f_\lambda\subseteq\mbox{im}\Psi_\lambda$ captures the parts of $\mathcal K$ and $\mbox{im}\Psi_\lambda\subseteq H_\lambda$ captures the parts of $\mathcal K$.
        Thus, we may define, in an analogous way to the one in Proposition \ref{prop:rayl}, a strong immersion $\Psi:G\to H_n^\alpha$.

        \begin{aff}\label{aff:capp}
        The subgraph $\mbox{im}\Psi\subseteq H_n^\alpha$ captures the parts of $\mathcal K$.
        \end{aff}
        Consider an immersion $\phi:\widetilde K\to H_n^{\alpha}$ with $\widetilde K\subseteq K$ of an element $K\in\mathcal K$.
        List $\lambda_1,...,\lambda_l$ all $\lambda<\mu$ such that $\mbox{im}\phi$ meets $H_{f_\lambda}\times\{\lambda\}$.
        Let $\overline \phi$ be the restriction of $\phi$ to $\overline K=\widetilde K-\phi^{-1}\left[\left(H_{f_{\lambda_1}}\times\{\lambda_1\}\right)\cup...\cup \left(H_{f_{\lambda_l}}\times\{\lambda_l\}\right)\right]$, and take its representative to $F_n$ denoted by $\overline{\phi^\prime}$.
        Define $\phi_0:\widetilde K\to F_n\cup \left(H_{f_{\lambda_1}}\times\{\lambda_1\}\right)\cup...\cup \left(H_{f_{\lambda_l}}\times\{\lambda_l\}\right)$ a representative of $\phi$, by taking $\overline{\phi^\prime}(v)$ if $v\in V(\overline K)$ and $\phi(v)$ otherwise.

        Consider $\overline \phi_0$ the restriction of $\phi_0$ to $\phi_0^{-1}\left[F_n\cup\left(H_{f_{\lambda_1}}\times\{\lambda_1\}\right)\right]\subseteq\widetilde K$ and take the representative of it to $\Psi[\overline C_{\lambda_1}]\subseteq F_n\cup \left(H_{f_{\lambda_1}}\times\{\lambda_1\}\right)\cup...\cup \left(H_{f_{\lambda_l}}\times\{\lambda_l\}\right)$.
        We may define $\phi_1$ a representative of $\phi_0$ to 
        \[F_n\cup \Psi[C_{\lambda_1}]\cup\left(H_{f_{\lambda_2}}\times\{\lambda_2\}\right)\cup...\cup \left(H_{f_{\lambda_l}}\times\{\lambda_l\}\right).\]

        Repeating this argument inductively, we may get $\phi_l$, a representative of $\phi$ to
        \[F_n\cup\Psi[C_{\lambda_1}]\cup...\cup\Psi[C_{\lambda_l}].\]
        In fact, from the construction, $\phi_l:\widetilde K\to\mbox{im}\Psi$ is a representative of $\phi$.
        This concludes the proof the Claim \ref{aff:capp}.

        Therefore the family $\mathcal F_\kappa^\alpha(\mathcal K)$ satisfies the desired properties.        
    \end{proof}

    Tying together Example \ref{exp:Kfree}, Remark \ref{rem:Kfree} and Proposition \ref{prop:Kfree}, Theorem \ref{thm:fin} follows, since for the case there is no subdivision of a star on the family $\mathcal K$ we can take $\mathcal F_\kappa(\mathcal K)=\bigcup_{\alpha<\kappa^+}\mathcal F_\kappa^\alpha(\mathcal K)$, otherwise, there is an $n<\omega$ such that $\mathbf A_\kappa(\mathcal K)=\mathbf A_\kappa^n(\mathcal K)$, which has a strongly universal family $\mathcal F_\kappa^n(\mathcal K)$ of cardinality at most countable.

    Yet, it is unclear what happens to the complexity when we allow for the family $\mathcal K$ to be infinite, both in terms of characterizing $\mathcal K$ for when the complexity has to be smaller or when it is higher, in fact, it is open the following question.

    \begin{quest}
        Is it consistent with ZFC that there is a family $\mathcal K$ of finite graphs such that 
        \[\Cpx(\mathbf A_{\aleph_0}(\mathcal K))>\aleph_1 ?\]
    \end{quest}

    In Section \ref{sec:Other}, some infinite families of forbidden graphs will be shown to have $\kappa^+$ strong complexity, mostly due to emerging properties of the class of them forbidden.
\section{An Infinite Case of High Complexity}\label{sec:Infinite}
    
    While the last section explored the case of finite forbidden subgraphs, the present section will be dedicated to study the case when the forbidden subgraph is infinite.
    Unlike the remarks made at the end of last section about the open question about infinite families of finite graphs, there is in fact a countable graph $K$ such that the class of countable rayless graphs with it forbidden has always complexity $\mathfrak c$.
    Thus, it is consistent that the complexity is strictly larger than $\aleph_1$.

    \begin{exemp}\label{ex:inf}
        There is a countable rayless graph $K$ of rank $1$ such that $\Cpx(\mathbf A_{\aleph_0}(K))=\mathfrak c$.
    \end{exemp}
    \begin{proof}
    The forbidden graph, denoted by $K$ (see Figure \ref{fig:ografo}), is defined by taking one cycle of each size, all disjoint and a new vertex --- which will be denoted by $w$ --- and creating an edge between it and exactly one vertex of each cycle.

    It follows from construction that it is a countable rayless graph of rank 1.
    Now, a family of graphs in $\mathbf A_{\aleph_0}(K)$ will be defined in a way that assuming that the complexity is smaller than $\mathfrak c$ would lead to a contradiction.

    Consider an infinite proper subset $X\subseteq\mathbb N_{\geq3}$ and a natural number $n\in X$.
    Take the cycles $C_n$ and $C_i$ for each $i\in\mathbb N_{\geq3}\setminus X$ such that $i<n$.
    Also take for each of these cycles one vertex $v_i\in V(C_i)$.
    Define the finite graph $G_X(n)$ as the gluing of the cycles above via the identification of the fixed vertices (see Figure \ref{fig:famdiffin}).

     For each infinite proper subset $X\subseteq\mathbb N_{\geq3}$, define a countable graph $G_X$ of rank 1 (see Figure \ref{fig:famdif}) by a similar construction to that of the graph $K$, with a center vertex denoted by $v_x$ connected via an edge to the center vertex of a copy of $G_X(n)$, for each $n\in X$.

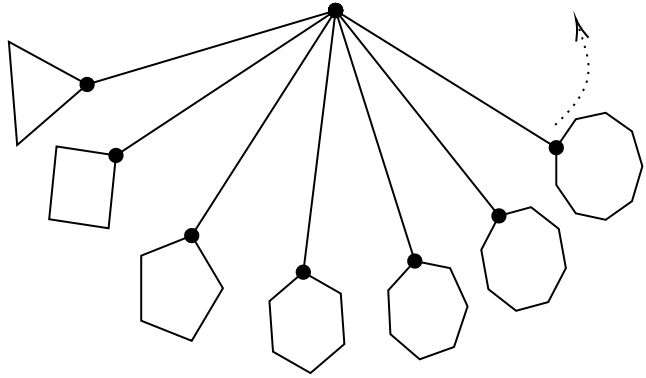
\begin{figure}
	\centering
	\tikzset{every picture/.style={line width=0.75pt}} 

\begin{tikzpicture}[x=0.75pt,y=0.75pt,yscale=-1,xscale=1]

\draw    (224.8,3.83) -- (99.47,41.16) ;
\draw [shift={(99.47,41.16)}, rotate = 163.42] [color={rgb, 255:red, 0; green, 0; blue, 0 }  ][fill={rgb, 255:red, 0; green, 0; blue, 0 }  ][line width=0.75]      (0, 0) circle [x radius= 3.35, y radius= 3.35]   ;
\draw [shift={(224.8,3.83)}, rotate = 163.42] [color={rgb, 255:red, 0; green, 0; blue, 0 }  ][fill={rgb, 255:red, 0; green, 0; blue, 0 }  ][line width=0.75]      (0, 0) circle [x radius= 3.35, y radius= 3.35]   ;
\draw   (99.47,41.16) -- (64.2,71.73) -- (60,19.62) -- cycle ;
\draw    (224.8,3.83) -- (114.02,76.99) ;
\draw [shift={(114.02,76.99)}, rotate = 146.56] [color={rgb, 255:red, 0; green, 0; blue, 0 }  ][fill={rgb, 255:red, 0; green, 0; blue, 0 }  ][line width=0.75]      (0, 0) circle [x radius= 3.35, y radius= 3.35]   ;
\draw [shift={(224.8,3.83)}, rotate = 146.56] [color={rgb, 255:red, 0; green, 0; blue, 0 }  ][fill={rgb, 255:red, 0; green, 0; blue, 0 }  ][line width=0.75]      (0, 0) circle [x radius= 3.35, y radius= 3.35]   ;
\draw   (114.02,76.99) -- (110.3,113.74) -- (80.33,109.19) -- (84.04,72.44) -- cycle ;
\draw    (224.8,3.83) -- (152.23,117.44) ;
\draw [shift={(152.23,117.44)}, rotate = 122.57] [color={rgb, 255:red, 0; green, 0; blue, 0 }  ][fill={rgb, 255:red, 0; green, 0; blue, 0 }  ][line width=0.75]      (0, 0) circle [x radius= 3.35, y radius= 3.35]   ;
\draw [shift={(224.8,3.83)}, rotate = 122.57] [color={rgb, 255:red, 0; green, 0; blue, 0 }  ][fill={rgb, 255:red, 0; green, 0; blue, 0 }  ][line width=0.75]      (0, 0) circle [x radius= 3.35, y radius= 3.35]   ;
\draw   (167.97,144) -- (152.23,170.56) -- (126.75,160.41) -- (126.75,127.58) -- (152.23,117.44) -- cycle ;
\draw    (224.8,3.83) -- (208.51,135.9) ;
\draw [shift={(208.51,135.9)}, rotate = 97.03] [color={rgb, 255:red, 0; green, 0; blue, 0 }  ][fill={rgb, 255:red, 0; green, 0; blue, 0 }  ][line width=0.75]      (0, 0) circle [x radius= 3.35, y radius= 3.35]   ;
\draw [shift={(224.8,3.83)}, rotate = 97.03] [color={rgb, 255:red, 0; green, 0; blue, 0 }  ][fill={rgb, 255:red, 0; green, 0; blue, 0 }  ][line width=0.75]      (0, 0) circle [x radius= 3.35, y radius= 3.35]   ;
\draw   (229.21,172.18) -- (212.12,186.83) -- (193.22,176.02) -- (191.42,150.55) -- (208.51,135.9) -- (227.41,146.72) -- cycle ;
\draw    (224.8,3.83) -- (264.73,130.28) ;
\draw [shift={(264.73,130.28)}, rotate = 72.48] [color={rgb, 255:red, 0; green, 0; blue, 0 }  ][fill={rgb, 255:red, 0; green, 0; blue, 0 }  ][line width=0.75]      (0, 0) circle [x radius= 3.35, y radius= 3.35]   ;
\draw   (284.52,173.66) -- (267.2,179.89) -- (252.43,167.18) -- (251.33,145.11) -- (264.73,130.28) -- (282.54,133.87) -- (291.35,153.18) -- cycle ;
\draw    (224.8,3.83) -- (307.17,107.5) ;
\draw [shift={(307.17,107.5)}, rotate = 51.53] [color={rgb, 255:red, 0; green, 0; blue, 0 }  ][fill={rgb, 255:red, 0; green, 0; blue, 0 }  ][line width=0.75]      (0, 0) circle [x radius= 3.35, y radius= 3.35]   ;
\draw [shift={(224.8,3.83)}, rotate = 51.53] [color={rgb, 255:red, 0; green, 0; blue, 0 }  ][fill={rgb, 255:red, 0; green, 0; blue, 0 }  ][line width=0.75]      (0, 0) circle [x radius= 3.35, y radius= 3.35]   ;
\draw   (340.93,133.87) -- (332.02,151.01) -- (315.83,155.4) -- (301.84,144.48) -- (298.26,124.64) -- (307.17,107.5) -- (323.36,103.1) -- (337.34,114.03) -- cycle ;
\draw    (224.8,3.83) -- (336.09,73.08) ;
\draw [shift={(336.09,73.08)}, rotate = 31.89] [color={rgb, 255:red, 0; green, 0; blue, 0 }  ][fill={rgb, 255:red, 0; green, 0; blue, 0 }  ][line width=0.75]      (0, 0) circle [x radius= 3.35, y radius= 3.35]   ;
\draw   (379.51,82.46) -- (374.27,100.09) -- (361.01,109.47) -- (345.93,106.22) -- (336.09,91.84) -- (336.09,73.08) -- (345.93,58.71) -- (361.01,55.45) -- (374.27,64.83) -- cycle ;
\draw  [dash pattern={on 0.84pt off 2.51pt}]  (335.5,61.58) .. controls (359.18,39.81) and (352.93,29.65) .. (346.59,9.94) ;
\draw [shift={(346,8.08)}, rotate = 72.73] [color={rgb, 255:red, 0; green, 0; blue, 0 }  ][line width=0.75]    (10.93,-3.29) .. controls (6.95,-1.4) and (3.31,-0.3) .. (0,0) .. controls (3.31,0.3) and (6.95,1.4) .. (10.93,3.29)   ;

\end{tikzpicture}
	\caption{The Graph $K$}
	\label{fig:ografo}
\end{figure}

\begin{figure}
    \centering
    \begin{minipage}{0.45\textwidth}
        \centering
                \tikzset{every picture/.style={line width=0.75pt}} 

\begin{tikzpicture}[x=0.75pt,y=0.75pt,yscale=-1,xscale=1]
	
	\draw  [color={rgb, 255:red, 208; green, 2; blue, 27 }  ,draw opacity=1 ] (95.31,84.4) -- (0,80.3) -- (51.45,5.16) -- cycle ;
	\draw  [color={rgb, 255:red, 208; green, 2; blue, 27 }  ,draw opacity=1 ] (127.88,117.23) -- (113.77,160.03) -- (67.08,170) -- (34.5,137.17) -- (48.61,94.37) -- (95.31,84.4) -- cycle ;
	\draw  [color={rgb, 255:red, 74; green, 144; blue, 226 }  ,draw opacity=1 ] (214,59.06) -- (190.49,104.71) -- (137.67,115.99) -- (95.31,84.4) -- (95.31,33.73) -- (137.67,2.13) -- (190.49,13.41) -- cycle ;
	\draw  [fill={rgb, 255:red, 0; green, 0; blue, 0 }  ,fill opacity=1 ] (92.43,84.4) .. controls (92.43,82.91) and (93.72,81.71) .. (95.31,81.71) .. controls (96.89,81.71) and (98.18,82.91) .. (98.18,84.4) .. controls (98.18,85.88) and (96.89,87.09) .. (95.31,87.09) .. controls (93.72,87.09) and (92.43,85.88) .. (92.43,84.4) -- cycle ;

\end{tikzpicture}
        \caption{$G_X(7)$ for $X=\{4,5,7,9,...\}$}
        \label{fig:famdiffin}
    \end{minipage}\hfill
    \begin{minipage}{0.48\textwidth}
        \centering
        \tikzset{every picture/.style={line width=0.75pt}} 

\begin{tikzpicture}[x=0.75pt,y=0.75pt,yscale=-1,xscale=1]

\draw  [color={rgb, 255:red, 208; green, 2; blue, 27 }  ,draw opacity=1 ] (33.01,84.82) -- (0,93.84) -- (6.05,60.19) -- cycle ;
\draw  [color={rgb, 255:red, 74; green, 144; blue, 226 }  ,draw opacity=1 ] (33.01,84.82) -- (52.12,105.91) -- (33.01,126.99) -- (13.9,105.91) -- cycle ;
\draw  [color={rgb, 255:red, 208; green, 2; blue, 27 }  ,draw opacity=1 ] (69.28,162.46) -- (31.85,153.99) -- (48.04,130.65) -- cycle ;
\draw  [color={rgb, 255:red, 74; green, 144; blue, 226 }  ,draw opacity=1 ] (77.53,183.45) -- (59.3,196.71) -- (39.79,183.93) -- (45.95,162.76) -- (69.28,162.46) -- cycle ;
\draw  [color={rgb, 255:red, 208; green, 2; blue, 27 }  ,draw opacity=1 ] (137.14,125.15) -- (138.54,162.68) -- (105.87,147.32) -- cycle ;
\draw  [color={rgb, 255:red, 208; green, 2; blue, 27 }  ,draw opacity=1 ] (111.32,193.87) -- (104.19,174.64) -- (117.81,159.05) -- (138.54,162.68) -- (145.66,181.91) -- (132.05,197.5) -- cycle ;
\draw [color={rgb, 255:red, 0; green, 0; blue, 0 }  ,draw opacity=1 ] [dash pattern={on 0.84pt off 2.51pt}]  (264.67,92.63) .. controls (289.9,81.08) and (301.46,41.72) .. (291.36,17.8) ;
\draw [shift={(290.54,15.99)}, rotate = 64.11] [color={rgb, 255:red, 0; green, 0; blue, 0 }  ,draw opacity=1 ][line width=0.75]    (10.93,-3.29) .. controls (6.95,-1.4) and (3.31,-0.3) .. (0,0) .. controls (3.31,0.3) and (6.95,1.4) .. (10.93,3.29)   ;
\draw  [color={rgb, 255:red, 74; green, 144; blue, 226 }  ,draw opacity=1 ] (180.87,169.67) -- (166.61,182.01) -- (147.78,178.9) -- (138.54,162.68) -- (145.86,145.57) -- (164.22,140.45) -- (179.8,151.17) -- cycle ;
\draw  [color={rgb, 255:red, 208; green, 2; blue, 27 }  ,draw opacity=1 ] (228.96,126.52) -- (192.02,133.35) -- (204.39,98.92) -- cycle ;
\draw  [color={rgb, 255:red, 208; green, 2; blue, 27 }  ,draw opacity=1 ] (239.29,161.31) -- (218.57,167.09) -- (203.04,152.58) -- (208.24,132.3) -- (228.96,126.52) -- (244.48,141.03) -- cycle ;
\draw  [color={rgb, 255:red, 208; green, 2; blue, 27 }  ,draw opacity=1 ] (261.57,155.04) -- (249.39,168.13) -- (231.22,169.04) -- (217.71,157.23) -- (216.78,139.62) -- (228.96,126.52) -- (247.13,125.62) -- (260.64,137.43) -- cycle ;
\draw  [color={rgb, 255:red, 74; green, 144; blue, 226 }  ,draw opacity=1 ] (263.57,108.51) -- (263.15,121.73) -- (254.06,131.6) -- (240.56,133.49) -- (228.96,126.52) -- (224.7,113.96) -- (229.76,101.68) -- (241.78,95.42) -- (255.13,98.12) -- cycle ;
\draw    (167.19,11) -- (33.01,84.82) ;
\draw [shift={(33.01,84.82)}, rotate = 151.18] [color={rgb, 255:red, 0; green, 0; blue, 0 }  ][fill={rgb, 255:red, 0; green, 0; blue, 0 }  ][line width=0.75]      (0, 0) circle [x radius= 3.35, y radius= 3.35]   ;
\draw [shift={(167.19,11)}, rotate = 151.18] [color={rgb, 255:red, 0; green, 0; blue, 0 }  ][fill={rgb, 255:red, 0; green, 0; blue, 0 }  ][line width=0.75]      (0, 0) circle [x radius= 3.35, y radius= 3.35]   ;
\draw    (167.19,11) -- (69.28,162.46) ;
\draw [shift={(69.28,162.46)}, rotate = 122.88] [color={rgb, 255:red, 0; green, 0; blue, 0 }  ][fill={rgb, 255:red, 0; green, 0; blue, 0 }  ][line width=0.75]      (0, 0) circle [x radius= 3.35, y radius= 3.35]   ;
\draw [shift={(167.19,11)}, rotate = 122.88] [color={rgb, 255:red, 0; green, 0; blue, 0 }  ][fill={rgb, 255:red, 0; green, 0; blue, 0 }  ][line width=0.75]      (0, 0) circle [x radius= 3.35, y radius= 3.35]   ;
\draw    (167.19,11) -- (138.54,162.68) ;
\draw [shift={(138.54,162.68)}, rotate = 100.7] [color={rgb, 255:red, 0; green, 0; blue, 0 }  ][fill={rgb, 255:red, 0; green, 0; blue, 0 }  ][line width=0.75]      (0, 0) circle [x radius= 3.35, y radius= 3.35]   ;
\draw [shift={(167.19,11)}, rotate = 100.7] [color={rgb, 255:red, 0; green, 0; blue, 0 }  ][fill={rgb, 255:red, 0; green, 0; blue, 0 }  ][line width=0.75]      (0, 0) circle [x radius= 3.35, y radius= 3.35]   ;
\draw    (167.19,11) -- (228.96,126.52) ;
\draw [shift={(228.96,126.52)}, rotate = 61.87] [color={rgb, 255:red, 0; green, 0; blue, 0 }  ][fill={rgb, 255:red, 0; green, 0; blue, 0 }  ][line width=0.75]      (0, 0) circle [x radius= 3.35, y radius= 3.35]   ;

\end{tikzpicture}
        \caption{Example for $X=\{4,5,7,9,...\}$}
        \label{fig:famdif}
    \end{minipage}
\end{figure}

    Towards a contradiction, assume there is a universal family of countable graphs $\mathcal H$ with less than $\mathfrak c $ elements.
    Take an almost disjoint family $\mathcal X$ with cardinality $|\mathcal H|^+$
    and for each $X\in\mathcal X$, take an immersion $\Phi_X:G_X\to H_X$, with $H_X\in\mathcal H$.
    Through arguments of cardinality, we may assume that all $H_X$ are the same $H$ and there is $v_0\in V(H)$ such that $\Phi_X(v_x)=v_0$ for every $X\in\mathcal X$.

    Fix two distinct elements $X,Y\in\mathcal X$ and denote their finite intersection as $Z=X\cap Y$.
    We will construct recursively immersions $\Psi_m$ of the subgraph
    \[\{w\}\cup\left(\bigcup_{i\in Z}C_i\right)\cup\left(\bigcup_{j<m}C_j\right)\subseteq K\]
    into $H$ such that $\Psi_m\subseteq\Psi_{m+1}$.

    Start by taking $\Psi_0$ as the restriction of $\Phi_X$ to $v_x$ and the $C_n\subseteq G_X(n)$ for each $n\in Z$.
    Assuming $\Psi_m$, if $m\in Z$, then $\Psi_{m+1}=\Psi_m$. 
    Otherwise, without loss of generality, $m\not\in X$, therefore there is $n\in X$ such that $n\geq m$ and $\Phi_X[G_X(n)]$ is disjoint from the image of $\Psi_m$, which is finite.
    Thus, there is a copy of $C_m$ in $G_X(n)$ through which we may extend $\Psi_m$ to $\Psi_{m+1}$.

    Taking the limit of this construction we end end with $\Psi=\bigcup_{n\in\mathbb N}\Psi_n:K\to H$ an immersion.
    This concludes the contradiction, and therefore there can not be a universal family with less than $\mathfrak c$ elements.
    \end{proof}

    It is interesting to remark that the arguments presented on Example \ref{ex:inf} are solely around the cardinality, and could be replicated for any cardinal lower than $\mathfrak c$, as is stated in the result below.
    \begin{teo}
    	There is a countable rayless graph $K$ of rank 1 such that for every infinite cardinal $\kappa$
    	\[\Cpx(\mathbf A_\kappa(K))\geq\mathfrak{c}.\]
    \end{teo}
	
	As a result, considering the negation of CH, as long as $\kappa$ is an infinite cardinal with $\kappa^+<\mathfrak{c}$, the theorem above states that even a countable graph of rank 1 being forbidden is enough to make the complexity high.

\section{Forbidding Infinite Bouquets}
	There are some rayless graphs $K$ with a kernel set $N$ such that every connected component of $K-N$ is isomorphic and connects the same way to $F$.
	Examples of this type of graphs include the complete bipartite graphs $K_{n,\kappa}$ with $n$ a natural number and $\kappa$ any infinite cardinal, as well as the bouquets of finite complete graphs and bouquets of cycles.	
	The classes of those subgraphs forbidden were studied (see \cite{KomPa2} and \cite{Kom}).
	Here we will see that these prohibitions result in small complexities in the countable rayless case due to the aforementioned property, formally defined below.

	\begin{defn}
		Let $K$ be a graph and $N\subseteq V(K)$ be a set of vertices of $K$ such that $K-N$ is connected.
		The bouquet of $K$ with stem $N$, denoted by $B(K,N)$, is defined by taking countably many pairwise disjoint copies fo $K$ and gluing them together via the identification of their copies of $N\subseteq K$ (see Figure \ref{fig:bouquet}). 
	\end{defn}
	\begin{figure}[ht]
		\centering
			\includegraphics[width=0.8\textwidth]{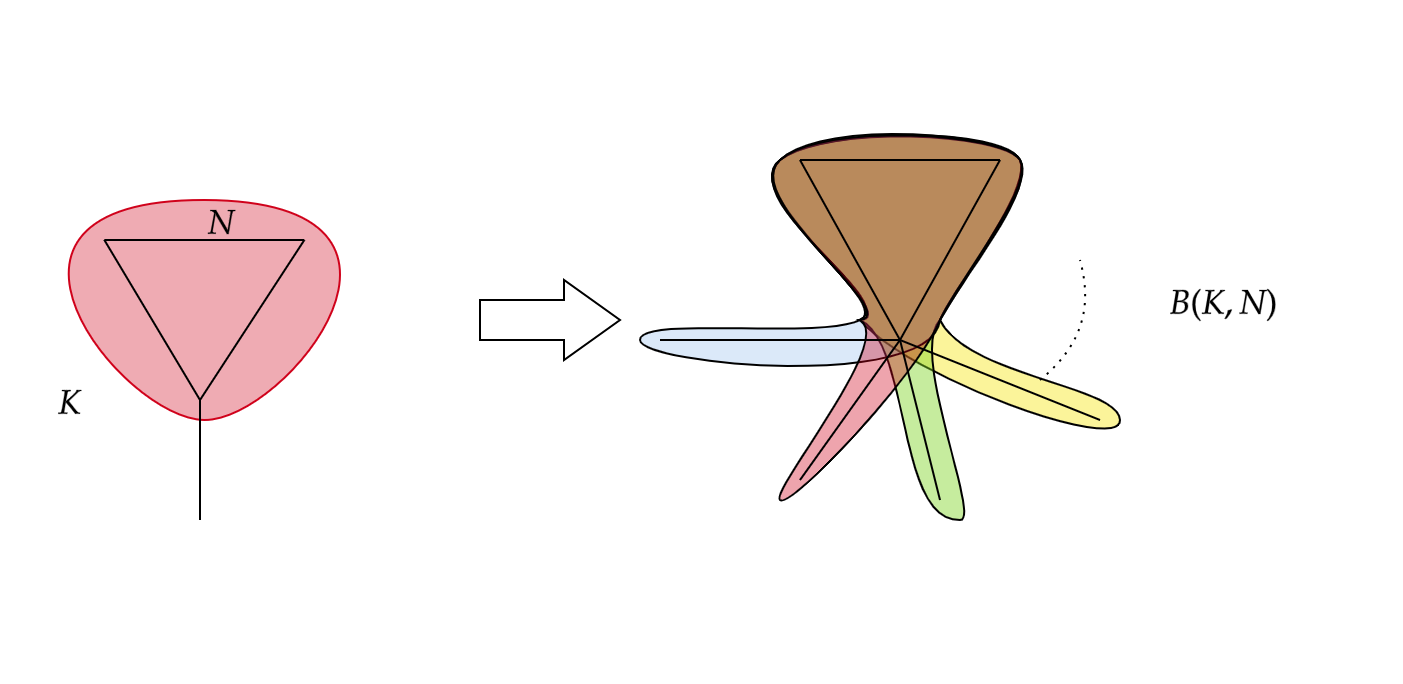}
			\caption{An example of a bouquet.}
			\label{fig:bouquet}
	\end{figure}
	\begin{rem}\hfill
		\begin{itemize}
			\item If $K$ is finite, then its bouquet with any stem is a rayless graph of rank 1, and $N$ is the minimal kernel.
			\item For each natural number $n$, the complete bipartite graph $K_{n,\aleph_0}$ is a bouquet of finite graphs.
			\item The usually known bouquets of cycles and bouquet of complete graphs are bouquets with singleton stems in this definition.
		\end{itemize}
	\end{rem}

	Except for the cases where those bouquets generate a subdivision of the infinite star, as was said after Remark \ref{rem:Kfree}, the lower bound of $\Cpx(\mathbf A_\kappa(B(K,N)))$ remains (sharply) at $\kappa^+$, for every infinite cardinal $\kappa$ and finite graph $K$.

	Let us consider the following characterization of a graph being a bouquet:
\begin{prop}
	A countable rayless graph $B$ of rank 1 is a isomorphic to a bouquet $B(K,N)$ if, and only if, there is a separator $N^\prime\subseteq B$ isomorphic to $N$ via $s:N^\prime\to N$ such that $B-N^\prime$ has countably many connected components, for each of them $C\subseteq B-N^\prime$, there is an isomorphism 
	\[i_C:C\cup N^\prime\to K,\]
	whose restriction to $N^\prime$ is $i\restriction_{N^\prime}=s:N^\prime\to N$. 
\end{prop}
	\begin{proof}
		Suppose that the graph $B$ is isomorphic to the bouquet $B(K,N)$ via the isomorphism $\phi:B(K,N)\to B$.
		Define $N^\prime=\phi[N]$ with the isomorphism $s:N^\prime\to N$ given by the restriction of the inverse of $\phi$.
		For each connected component $C\subseteq B-N^\prime$, note that it is isomorphic to a connected component of $B(K,N)-N$.
		It follows from the definition of $B(K,N)$ that $\phi^{-1}[C]\cup N$ is isomorphic to $K$ with $N$ rigid.
		Therefore we may define isomorphisms $i_C$ as in the statement of the proposition.
		
		For the other implication, consider in $B$ the separator $N^\prime$, as well as the isomorphisms $s:N^\prime\to N$ and $i_C:C\cup N^\prime\to K$.
		Enumerate the connected components of $B-N^\prime$ as $\{C_n:n\in\mathbb N\}$.
		Also enumerate the connected components of $B(K,N)-N$ as $\{H_n:n\in\mathbb N\}$.
		As was stated above, there is an isomorphism $j_n:K\to H_n\cup N$ with $j_n\restriction_N=\id_N$.
		
		Define the following morphism:
		\begin{center}
		\begin{tabular}{rcl}
			$\phi:B$ & $\to$ & $B(K,N)$\\
			$v$  & $\mapsto$ & $\begin{cases}s(v)\text{, if }v\in V(N^\prime)\\ j_n\circ i_{C_n}(v)\text{, if }v\in V(C_n)\text{ for some }n\in\mathbb N\end{cases}$
		\end{tabular}.
		\end{center}
		
		Note that it is in fact an isomorphism.
	\end{proof}

	In order to construct a universal family we will define a stronger version of a subgraph capturing parts of another graph.
	
	\begin{defn}
		Given two graphs $K$ and $G$, a subgraph $F\subseteq G$ captures $K$ with $X\subseteq F$ rigid if for every immersion $\phi:K\to G$ there is a representative $\phi^\prime$ of $\phi$ in $F$ such that ${\phi^\prime}^{-1}[X]=\phi^{-1}[X]$, where they coincide.
	\end{defn} 
	
	 As in the case for the weaker notion of capturing, we can always consider that there is a finite superset capturing $K$, with it rigid.
	 
	 \begin{lem}
	 	Let $G$ be a rayless graph and $K$ a finite graph.
	 	For any pair of finite subgraphs $X\subseteq F\subseteq G$, there is a finite subgraph $\widetilde F\subseteq G$ such that $F\subseteq\widetilde F$ and it captures the parts of $K$ on $G$ with $X$ rigid.
	 \end{lem} 
	 
	 The proof of the lemma above follows analogously to the proof of Lemma \ref{lem:capture}, with the only difference being that, when choosing the finitely many connected components to construct $H$ in the induction step, the immersions chosen take into consideration the rigidity of $X$.
	 
	 \begin{prop}
	 	Let $K$ be a connected finite graph with $N\subseteq V(K)$ such that $K-N$ is connected.
	 	For every cardinal $\kappa$ and non-zero ordinal number $\alpha<\kappa^+$ there is a family $\mathcal F^\alpha_\kappa(B(K,N))\subseteq\mathbf A_\kappa^\alpha(B(K,N))$ of cardinality up to $\kappa$ such that, for every $G\in A_\kappa^\alpha(B(K,N))$ and finite set of vertices $X\subseteq V(G)$, there is $H\in\mathcal F^\alpha_\kappa(B(K,N))$ and a strong immersion $\Phi:G\to H$ whose image $\im\Phi\subseteq H$ captures the parts of $K$ with $\Phi[X]$ rigid.
	 \end{prop}
	 \begin{proof}
	 	This proof follows by transfinite induction, with the cases $\alpha=1$ and $\alpha$ a limit ordinal analogous to those in Proposition \ref{prop:rayl}.
	 	
	 	For $\alpha=\beta+1$.
	 	For each $n\in\mathbb N$ and finite set of vertices $X\subseteq V(F_n)$, consider $\mathcal I(F_n,X;K)$ the family of all immersions $f:F_n\to H$, to any $H\in\mathcal F_\kappa^\alpha(B(K,N))$, whose image $\im f\subseteq H$ captures $K$ with $f[X]$ rigid.
	 	
	 	Consider the subfamily $\mathcal I^\ast(F_n,X;K)\subseteq\mathcal I(F_n,X;K)$ the strong immersions $f:F_n\to H$ such that there is an immersion $\iota:K\to H$ with $\iota^{-1}\left[f[F_n]\right]=N$.
	 	Denote the complementary subfamily as $\mathcal I^\prime(F_n,X;K)=\mathcal I(F_n,X;K)\setminus\mathcal I^\ast(F_n,X;K)$.
	 	
	 	For each triple $n\in\mathbb N$, $X\subseteq V(F_n)$ and $h\in\{\emptyset\}\cup\mathcal I^\ast(F_n,X;K)$ we will define the graph $H^\alpha_{n,X,h}$ similarly to $H^\alpha_n$ defined in Proposition \ref{prop:rayl}.
	 	Except for the difference that the immersions we consider range in $f\in\mathcal I^\prime(F_n,X;K)$ and, if $h\not=\emptyset$, consider also a single copy of $H_h$, instead of $\kappa$ many.
	 	
	 	Notice that $F_n$ is a minimal kernel of $H^\alpha_{n,X,h}$ that verify that it is a $\kappa$-rayless graph of rank up to $\beta$.
	 	Suppose towards a contradiction that there is an immersion $\eta:B(K,N)\to H^\alpha_{n,X,h}$.
	 	Denote by $\partial N^\prime$ the subset of $N^\prime$ of all neighbors of any vertex in $B(K,N)-N^\prime$. 
	 	By definition, there is a connected component $C\subseteq H^\alpha_{n,X,h}$ such that $\eta[\partial N^\prime]\subseteq C\cup F_n$.
	 	If it isn't contained in $F_n$, then all but finitely many connected components of $B(K,N)-N^\prime$ are contained also in $C\cup F_n$, which would imply that a $H\in\mathcal F_\kappa^\beta(B(K,N))$ is not $B(K,N)$-free.
	 	
	 	Thus, we may suppose that $\eta[N^\prime]\subseteq F_n$, what in turn implies that there is a copy of $B(K,N)$ in $H_h\cup F_n$, which is another contradiction.
	 	We conclude that in fact $H^\alpha_{n,X,h}$ is $B(K,N)$-free. 
	 	
	 	Let $G\in \mathbf A_\kappa^\alpha(B(K,N))$ and $X\subseteq V(G)$ finite.
	 	Consider $F\subseteq G$ a kernel of $G$ that contains $X$ and captures the parts of $K$ on $G$ with $X$ rigid.
	 	There is $n\in\mathbb N$ such that there is an isomorphism $\psi:F\to F_n$.
	 	There are only finitely many connected components $C$ of $G-F$ such that there is an immersion $\iota:K\to C\cup F$ with $\iota^{-1}[F]=N$.
	 	Consider $C^\prime$ their union, or $C^\prime=\emptyset$ if there is no such connected component.
	 	Consider also $\overline C^\prime=C^\prime\cup F$, the enumeration of all other connected components $\{C_\lambda:\lambda<\mu\}$ and for each $\lambda<\mu$ consider $\overline C_\lambda=C_\lambda\cup F$.
	 	
	 	For each $\lambda<\mu$, consider $H_\lambda\in\mathcal F^\beta_\kappa(B(K,N))$ together with a strong immersion $\Phi_\lambda:\overline C_\lambda\to H_\lambda$ whose image captures parts of $K$ with $\Phi_\lambda[F]$ rigid.
	 	Notice that for each $\lambda<\mu$ the image of the strong immersion $\Phi_\lambda\circ\psi^{-1}:F_n\to H_\lambda$ captures parts of $K$ with $\psi[X]$ rigid.
	 	Moreover, this strong immersion is in $\mathcal I^\prime(F_n,\psi[X];K)$, since there is no immersion $\iota:K\to \overline C_\lambda$ with $\iota^{-1}[F]=N$ and $\im\Phi_\lambda$ captures parts of $K$ on $H_\lambda$ with $\Phi_\lambda[F]$ rigid. 
	 	
	 	If $C^\prime=\emptyset$, then take $h=\emptyset$.
	 	Otherwise, consider $H^\prime\in\mathcal F^\beta_\kappa(B(K,N))$ together with a strong immersion $\Phi^\prime:\overline C^\prime\to H^\prime$ whose image captures parts of $K$ on $H^\prime$ with $\Phi^\prime[F]$ rigid.
	 	Take $h=\Phi^\prime\circ\psi^{-1}:F_n\to H^\prime$, which is in $\mathcal I^\ast(F_n,\psi[X];K)$.
	 	
	 	From the $\Phi_\lambda$ and $\Phi^\prime$ defined, there is a strong immersion $\Phi:G\to H^{\alpha}_{n,\psi[X],h}$ whose image captures parts of $K$ on $H^\alpha_{n,\psi[X],h}$ with $\psi[X]=\Phi[X]$ rigid, in an analogous manner to that of Proposition \ref{prop:Kfree}.
	 	This concludes the proof of the proposition.
	 \end{proof}
	 
	 The proposition above results in the following result.
	 
	 \begin{teo}
	 	Let $\kappa$ be an infinite cardinal, let $K$ be a finite graph with $N\subseteq V(K)$ such that $K-N$ is connected.
	 	Then $\Stc(\mathbf A_\kappa(B(K,N)))=\kappa^+$, if $B(K,N)$ is not a subdivision of a star.
	 \end{teo}

\section{Other Rayless Classes}\label{sec:Other}

    This section establishes for three cases of forbidding infinitely many finite subgraphs that the results of small complexity still holds up.
    When forbidding the family of all cycles, here denoted by $\mathcal C$, the resulting class is exactly the trees.
    When forbidding the family of all odd cycles, here denoted by $\mathcal C^o$, the resulting class is exactly the bipartite graphs.
    When forbidding the family of all even cycles, here denoted by $\mathcal C^e$, the resulting class is not so straightforward, although it has the property that no cycle admits an open ear.
    
    \begin{defn}
    	Given a graph $G$ and a subgraph $H\subseteq G$, an open ear of $H$ in $G$ is a path $P\subseteq G$ such that exactly only its end vertices are in $H$.
    	An ear is either an open ear, or a cycle in $G$ that meets $H$ in exactly one vertex.
    \end{defn}

    The proof of the two first classes follows the same way laid out in Section \ref{sec:ray}, except that when taking the universal family of rank 0, consider only those with the corresponding class forbidden.
    The third case is closer to that of Section \ref{sec:Fin}, which is to say, add restrictions to the immersions considered during the construction, together with an added parameter to the construction, instead of only the natural number $n$. 
    
    \begin{teo}
        For every cardinal $\kappa$ and non-zero ordinal number $\alpha<\kappa^+$,
        \[\Stc(\mathbf A_\kappa^\alpha(\mathcal C))=\begin{cases}
            \aleph_0\text{, if }\alpha\mbox{ is a successor ordinal;}\\
            \mbox{cf}(\alpha)\text{, if }\alpha\mbox{ is a limit ordinal.}
        \end{cases}.\]
        This concludes that $\Stc(\mathbf A_\kappa(\mathcal C))=\kappa^+$.
    \end{teo}
    \begin{proof}
        The proof of this fact is similar to that of Proposition \ref{prop:rayl}, except that at the first step $\alpha=1$ it is only considered the finite trees.
        The rest of the construction on top of it cannot construct cycles and the proof of strong universality also follows as there. 
    \end{proof}

    \begin{teo}
        For every cardinal $\kappa$ and non-zero ordinal number $\alpha<\kappa^+$,
        \[\Stc(\mathbf A_\kappa^\alpha(\mathcal C^o))=\begin{cases}
            \aleph_0\text{, if }\alpha\mbox{ is a successor ordinal;}\\
            \mbox{cf}(\alpha)\text{, if }\alpha\mbox{ is a limit ordinal.}
        \end{cases}.\]
        This concludes that $\Stc(\mathbf A_\kappa(\mathcal C^o))=\kappa^+$.
    \end{teo}
    \begin{proof}
        Follow the steps of the proof of Proposition \ref{prop:rayl}, except that for the fact that at the first step $\alpha=1$ it is only considered the finite bipartite graphs.
        That the graph $H_n^\alpha$ is bipartite follows from the fact that immersions of connected graphs preserve the bipartition, therefore the kernel set $F_n$ sets each side of each connected component to one of its two sides, completing a bipartition.
    \end{proof}

    \begin{teo}\label{Thm:even}
        For every cardinal $\kappa$ and every non-zero ordinal number $\alpha<\kappa^+$ there is a family $\mathcal F_\kappa^{\alpha}(\mathcal C^e)\subseteq\mathbf A_\kappa^\alpha(\mathcal C^e)$ such that 
        \begin{enumerate}
        \item For every $G\in\mathbf A_\kappa^\alpha(\mathcal C^e)$ and finite subgraph $K\subseteq G$ there is an $H\in\mathcal F_\kappa^\alpha(\mathcal C^e)$ and a strong immersion $\Phi:G\to H$ such that if there is no open ear of $K$ in $G$, then there is no open ear of $\Phi[K]$ in $H$;
        \item The cardinality  is $|\mathcal F_\kappa^\alpha(\mathcal C^e)|\leq\kappa$.
        \end{enumerate}
        This concludes that $\Stc(\mathbf A_\kappa(\mathcal C^e))=\kappa^+$.
    \end{teo}
    \begin{proof}
        For $\alpha=1$, take the family of every connected finite graph without even cycles $\mathcal F^1_\kappa(\mathcal C^e)$ and an enumeration $\{F_n:n\in\mathbb N\}$.
        For a limit ordinal $\alpha$, it is enough to consider $\mathcal F^\alpha_\kappa(\mathcal C^e)=\bigcup_{\gamma<\alpha}\mathcal F^\alpha_\kappa(\mathcal C^e)$.

        For $\alpha=\beta+1$, for every $n\in\mathbb N$ and $H\in\mathcal F^\beta_\kappa(\mathcal C^e)$, consider the following families of strong immersions
        \[\mathcal I^e(F_n,H)=\{h\in\mathcal I(F_n, H):\mbox{im} h\subseteq H\mbox{ has open ears}.\}\]
        and take
        \[\mathcal I^\ast(F_n,H)=\mathcal I(F_n,H)\setminus\mathcal I^e(F_n,H).\]

		Let $n\in\mathbb N$ be natural number and $h\in\{\emptyset\}\cup\bigcup\{\mathcal I^e(F_n,H):H\in\mathcal F_\kappa^\beta(\mathcal C^e)\}$.
		As in Proposition \ref{prop:piso}, define a graph $H^\alpha_{n,h}$, except that not all immersions $f$ are considered, only the ones in $\bigcup\{\mathcal I^\ast(F_n,H):H\in\mathcal F_\kappa^\beta(\mathcal C^e)\}$.
		Moreover, also consider exactly one copy of $H_h$, unless $h=\emptyset$, in which case, only the former are considered.


        The above defined graph is a $\kappa$-rayless graph of rank at most $\beta$.
        It follows from the construction, by avoiding open ears, that every cycle of $H^\alpha_{n,h}$ is contained in a subgraph of type $F_n\cup H_{f}\times\{\lambda\}$, which is isomorphic to $H\in\mathcal F^\beta_\kappa(\mathcal C^e)$.
        This concludes that there is no even cycle.
        The family $\mathcal F^\alpha_\kappa(\mathcal C^e)\subseteq \mathbf A^\alpha_\kappa(\mathcal C^e)$ is in fact strongly universal.

        For that end, take any graph $G\in\mathbf A^\alpha_\kappa(\mathcal C^e)$.
        Take also any finite subgraph $K\subseteq G$ that has no open ears.
        Consider a connected kernel set $F\subseteq G$ that contains $K$ and $n\in\mathbb N$ with an isomorphism $\psi:F_n\to F$.
        Since there are no even cycles, there are only finitely many connected components $C\subseteq G-F$ such that there is an open ear of $F$ on $C$.
        Consider the graph $C^\prime\subseteq G$ given by the union of all such connected components and $\overline C^\prime=C^\prime\cup F\subseteq G$, if there are any.
        Enumerate all the other connected components of $G-F$ by $\{C_\lambda:\lambda<\mu\}$ and define $\overline C_\lambda=C_\lambda\cup F\subseteq G$

        Consider an element $H^\prime\in\mathcal F^\beta_\kappa(\mathcal C^e)$ with a strong immersion $\Psi^\prime:\overline C^\prime\to H^\prime$ and for each $\lambda<\mu$, take an element $H_\lambda$ with a strong immersion $\Psi_\lambda:\overline C_\lambda\to H_\lambda$ such that $\Psi_\lambda[F]$ has no open ears on $H_\lambda$.
        All immersions considered above have the property that their image of $K$ has no open ears.
        Taking $h=\psi\circ\Psi^\prime:F_n\to H^\prime$ if $C^\prime$ is not empty and $h=\emptyset$ if it is, we can define similarly to Proposition \ref{prop:rayl} a strong immersion $\Phi:G\to H^\alpha_{n,h}$.
        Notice that $\Phi[K]$ has no open ears on $H^\alpha_{n,h}$, because an open ear of it would be an open ear of $\Phi[F]$, therefore it would be an open ear of $\Phi[K]$ in $F_n\cup H_h$, that contradicts the choice of $\Phi^\prime $ as $\Psi^\prime[K]$ having no open ears on $H^\prime$.
    \end{proof}

    One last example of class of rayless with still small complexity is the class of graphs without infinite trails.
    A trail, much like a path, is a sequence of connected edges on the graph.
    They may repeat a vertex, but never an edge.
    An infinite trail is the analogous to a ray, by having a starting point an going to the infinite.
    Every ray is an infinite trail, therefore the class of graph without infinite trails is a subclass to that of the rayless class.

    \begin{defn}
        The class of $\kappa$-graphs without trails is denoted as $\mathbf B_\kappa$, the subclass of rayless graphs of rank less than an ordinal $\alpha$ is $\mathbf B_\kappa^\alpha$.
    \end{defn}
    
    In fact, it is possible to characterize when a rayless graph has no infinite trails.
    
    \begin{lem}\label{lem:trails}
        A rayless graphs has no infinite trails if, and only if, there is no vertex $v$ with an infinite family of pairwise edge-disjoint cycles, each containing the vertex.
        In this case, the graph is finitely separated, that is, for every two vertices, there is a finite cut that separates them.
    \end{lem}

    First remark that in trees, an infinite trail is the same as a ray, concluding that their complexity cannot be less than $\kappa^+$.

    This characterization allows us to construct strongly universal families for graphs without trails of rayless rank up to a certain $\alpha$ similar to the one constructed in Theorem \ref{Thm:even}, and limiting the amount of closed ears of the kernel set.
    
    \begin{teo}\label{Thm:trail}
        For every cardinal $\kappa$ and non-zero ordinal number $\alpha<\kappa^+$ there is a family $\mathcal F_\kappa^{\alpha}\subseteq\mathbf B_\kappa^\alpha$ such that 
        \begin{enumerate}
        \item For every $G\in\mathbf B_\kappa^\alpha$ and finite subgraph $K\subseteq G$ there is $H\in\mathcal F_\kappa^\alpha$ and a strong immersion $\Phi:G\to H$ such that if there is no ear of $K$ in $G$, then there is no ear of $\Phi[K]$ in $H$;
        \item The cardinality  is $|\mathcal F_\kappa^\alpha|\leq\kappa$.
        \end{enumerate}
        This concludes that $\Stc(\mathbf B_\kappa)=\kappa^+$.
    \end{teo}
   \begin{proof}
        For $\alpha=1$, take the family of every connected finite graph $\mathcal F_\kappa^1$ and an enumeration $\{F_n:n\in\mathbb N\}$.
        For $\alpha$ a limit ordinal, it is enough to consider $\mathcal F_\kappa^\alpha=\bigcup_{\gamma<\alpha}\mathcal F_\kappa^\gamma$.

        For $\alpha=\beta+1$, for every $n\in\mathbb N$ and $H\in\mathcal F_\kappa^\beta$, consider the following families of strong immersions
        \[\mathcal I^o(F_n,H)=\{h\in\mathcal I(F_n,H):\mbox{im}h\subseteq H\mbox{ has ears.}\}\]
        and take
        \[\mathcal I^\ast(F_n, H)=\mathcal I(F_n,H)\setminus\mathcal I^o(F_n, H).\]

        Define for each pair $n\in\mathbb N$ and $h\in\{\emptyset\}\cup\bigcup_{H\in\mathcal F_\kappa^\beta}\mathcal I^0(F_n,H)$ the graph $H^\alpha_{n,h}$ as the one on Theorem \ref{Thm:even}.
        Towards a contradiction, suppose that there is a trail on a $H^\alpha_{n,h}$.
        From Lemma \ref{lem:trails}, there is a fixed vertex $v$ such that there is a pairwise edge-disjoint family of cycles $\{C_n:n\in\mathbb N\}$ with $v$ in each $C_n$.
        There are two possibilities for the position of $v$, the first one being $v\in V(F_n\cup H_h)$, but then every cycle is contained in $F_n\cup H_h$, which has no infinite trails, contradicting Lemma \ref{lem:trails}.
        The other possibility is $v\in H_f\times\{\lambda\}$ for some pair of $f\in\mathcal I^\ast(F_n,H)$ and $\lambda<\kappa$. 
        This case results in every cycle being contained in $H_f\times\{\lambda\}$, also contradicting Lemma \ref{lem:trails}, since it has no infinite trails.

        Take $G\in\mathbf B_\kappa^\alpha$ and a finite subgraph $K\subseteq G$.
        Let $F\subseteq G$ be a kernel set that contains $K$.
        If there are ears of $F$ on $G$, then they are contained in finitely many connected components of $G-F$, otherwise there would be an infinite internally disjoint family of ears with the same ends, which contains an infinite trail.
        Consider the subgraph $C^\prime\subseteq G-F$ the union of such components.
        Similarly to the construction in Theorem \ref{Thm:even}, we may define $n\in\mathbb N$ with $F$ isomorphic to $F_n$, $h$ and a strong immersion $\Phi:G\to H^\alpha_{n,h}$ that satisfies the property (1).
    \end{proof}

    \section*{Acknowledgements}
    The first named author thanks the support of Fundação de Amparo à Pesquisa do Estado de São Paulo (FAPESP), being sponsored through grant number 2025/12199-3. 
    The second named author acknowledges the support of Conselho Nacional de Desenvolvimento Científico e Tecnológico (CNPq) through grant number 141373/2025-3.
    This study was financed in part by the Coordenação de Aperfeiçoamento de Pessoal de Nível Superior – Brasil (CAPES) – Finance Code 001.

\end{document}